\documentclass{article}

\usepackage[final]{nips_2017}

\usepackage[utf8]{inputenc} 
\usepackage[T1]{fontenc}    
\usepackage{hyperref}       
\usepackage{url}            
\usepackage{booktabs}       
\usepackage{amsfonts}       
\usepackage{nicefrac}       
\usepackage{microtype}      
\usepackage{amsmath} 
\usepackage{graphicx}
\usepackage{amsthm}
\usepackage[dvipsnames,usenames]{color}

\title{How regularization affects the critical points in linear networks}

\author{
Amirhossein Taghvaei \\ Coordinated Science Laboratory \\University of Illinois at Urbana-Champaign \\ Urbana, IL, 61801 \\
Email: taghvae2@illinois.edu
\And
Jin W. Kim \\ Coordinated Science Laboratory \\University of Illinois at Urbana-Champaign \\ Urbana, IL, 61801 \\
Email: jkim684@illinois.edu
\And
Prashant G. Mehta \\ Coordinated Science Laboratory \\University of Illinois at Urbana-Champaign \\ Urbana, IL, 61801, \\
Email: mehtapg@illinois.edu
}

\def\Re{\mathbb{R}}

\def\A0{{\sf C}}

\def\tA0{\tilde{\sf C}}

\def\tlambda{\tilde{\lambda}}

\newcommand{\sfJ}{{\sf J}}

\usepackage{eufrak}

\def\Sec#1{Sec.~\ref{#1}}

\def\notes#1{\marginpar{\tiny #1}\typeout{Notes!
Notes!
Notes!
}}
\renewcommand{\notes}[1]{\typeout{notes!}}

\newcommand{\field}[1]{\mathbb{#1}}
\newcommand{\tr}{\mbox{tr}}

\def\Re{\field{R}}

\def\Sec#1{Sec.~\ref{#1}}

\def\tf{T}

\def\det{\text{det}}

\makeatletter

\newcommand{\Rom}[1]{\expandafter\@slowromancap\romannumeral #1@}
\makeatother

\newcounter{rmnum}

\newenvironment{romannum}{\begin{list}{{\upshape (\roman{rmnum})}}{\usecounter{rmnum}
\setlength{\leftmargin}{17pt}
\setlength{\rightmargin}{4pt}
\setlength{\itemindent}{-1pt}
}}{\end{list}}

\newcounter{anum}

\newcommand{\ud}{\,\mathrm{d}}

\def\E{{\sf E}}

\newcommand{\lr}[2]{\langle #1, #2 \rangle}

\newcommand{\expect}[1]{{\sf E}[#1]}
\newcommand{\newP}[1]{\noindent{\bf #1:}}

\usepackage{ulem}

\newtheorem{theorem}{Theorem}

\newtheorem{proposition}{Proposition}
\newtheorem{lemma}{Lemma}[section]
\newtheorem{definition}{Definition}
\newtheorem{example}{Example}
\newtheorem{remark}{Remark}

\makeatletter
\newsavebox\myboxA
\newsavebox\myboxB
\newlength\mylenA

\newcommand*\xoverline[2][0.75]{%
    \sbox{\myboxA}{$\m@th#2$}%
    \setbox\myboxB\null
    \ht\myboxB=\ht\myboxA%
    \dp\myboxB=\dp\myboxA%
    \wd\myboxB=#1\wd\myboxA
    \sbox\myboxB{$\m@th\overline{\copy\myboxB}$}
    \setlength\mylenA{\the\wd\myboxA}
    \addtolength\mylenA{-\the\wd\myboxB}%
    \ifdim\wd\myboxB<\wd\myboxA%
       \rlap{\hskip 0.5\mylenA\usebox\myboxB}{\usebox\myboxA}%
    \else
        \hskip -0.5\mylenA\rlap{\usebox\myboxA}{\hskip 0.5\mylenA\usebox\myboxB}%
    \fi}
\makeatother

\def\argmax{\mathop{\text{\rm arg\,max}}}

\newcommand{\noise}{\xi}
\newcommand{\tX}{\eta}
\newcommand{\Df}{\mathsf{D}f}

\begin{document}

\maketitle

\begin{abstract}
This paper is concerned with the problem of representing and learning
a linear transformation using a linear neural network.  In recent
years, there has been a growing interest in the study of such networks
in part due to the successes of deep learning.  The main question of
this body of research and also of this paper pertains to the existence and
optimality properties of the critical points of the mean-squared loss
function.  The primary concern here is the robustness of the critical
points with regularization of the loss
function.  An optimal control model is introduced for this purpose and
a learning algorithm (regularized form of backprop) derived for the same
using the Hamilton's formulation of optimal control.  The formulation
is used to provide a complete characterization of the critical points
in terms of the solutions of a nonlinear matrix-valued equation,
referred to as the characteristic equation.  Analytical and numerical
tools from bifurcation theory are used to compute the critical points
via the solutions of the characteristic equation.  The main conclusion
is that the critical point diagram can be fundamentally different even
with arbitrary small amounts of regularization.

\end{abstract}

\section{Introduction}
\label{sec:intro}
This paper is concerned with the problem of representing and learning
a linear transformation with a linear neural network.  
Although a classical problem~(\cite{baldi1989neural,baldi1995learning}),
there has been a renewed interest in such networks
(\cite{DBLP:journals/corr/HardtM16,DBLP:journals/corr/SaxeMG13,kawaguchi2016deep})
because of the successes of deep learning. 
A focus of the recent
research on these (and also nonlinear) networks has been on the analysis
of the critical points of the non-convex loss
function~(\cite{choromanska2015loss,choromanska2015open,dauphin2014identifying,soudry2016no}). This
is also the focus here.

\newP{Problem} The
input-output model is assumed to be of the following linear form:
\begin{equation}
Z = R X_0 + \noise
\label{eq:lin_model}
\end{equation}
where $X_0\in \Re^{d\times 1}$ is the input, $Z\in \Re^{d\times 1}$ is
the output, and $\noise\in \Re^{d\times 1}$ is the noise.  The input $X_0$ is
modeled as a random variable whose distribution is denoted as
$p_0$.  Its second moment is denoted as $\Sigma_0 = \E [X_0
X_0^\top]$ and assumed to be finite.  The noise $\noise$ is assumed to be
independent of $X_0$, with zero mean and finite variance. 
The linear transformation $R \in
M_d(\Re)$ is assumed to satisfy a property (P1) introduced in \Sec{sec:critical_points} ($M_d(\Re)$ denotes the set of 
$d\times d$ matrices).  The problem is to learn the weights of a
linear neural network from i.i.d. input-output samples $\{(X_0^k,Z^k)\}_{k=1}^K$.     

\newP{Solution architecture} is a continuous-time linear feedforward neural network model:
\begin{equation}\label{eq:lin_net}
\frac{\ud X_t}{\ud t} = A_t X_t 
\end{equation}
where $A_t\in M_d(\Re)$ are the network weights indexed by
continuous-time (surrogate for layer) $t\in [0,T]$, and $X_0$ is the initial
condition at time $t=0$ (same as the input data).  The parameter $T$
denotes the network depth.  The optimization problem is to choose the
weights $A_t$ over the time-horizon $[0,\tf]$ to minimize the
mean-squared loss function:
\begin{equation}
\E[|X_\tf - Z|^2] = \E[| X_\tf- R X_0|^2] \;+\; \E[|\noise|^2] 
\label{eq:loss_fn}
\end{equation}
This problem is referred to as the $[\lambda=0]$ problem.  

Backprop is a
stochastic gradient descent algorithm for learning the weights $A_t$.
 In general, one obtains (asymptotic) convergence of the learning algorithm to a (local) minima of the
optimization problem~\cite{2016arXiv160204915L,2015arXiv150302101G}.
This has spurred investigation of the critical points of
the loss function~\eqref{eq:loss_fn} and the optimality properties
(local vs. global minima, saddle points) of these points.  For linear
multilayer (discrete) neural networks (MNN), strong conclusions have
been obtained under rather mild conditions:
every local minimum is a global minimum and every critical point that
is not a local minimum is a saddle
point~\cite{kawaguchi2016deep,baldi1989neural}.  In
experiments, some of these properties are also observed empirically in
deep nonlinear networks; cf.,~\cite{choromanska2015open,dauphin2014identifying,DBLP:journals/corr/SaxeMG13}.  The discrete MNN counterpart of the continuous-time
model~\eqref{eq:lin_net} is the linear residual
network model of~\cite{DBLP:journals/corr/HardtM16}: An Euler
discretization of~\eqref{eq:lin_net} yields the residual 
network.  For such networks, it is shown
in~\cite{DBLP:journals/corr/HardtM16} that, in some neighborhood of
$A_t\equiv 0$, every critical point is a global minimum.

In this paper, the optimization problem is formulated as an optimal control problem:
\begin{equation}
\begin{aligned}
\underset{A}{\text{Minimize:}}&\quad {\sf J}[A]= \expect{\;\frac{\lambda}{2} \int_0^{\tf}
\text{tr}\,(A_t^\top\;A_t)\ud t \;+\; \frac{1}{2}|X_T-Z|^2\;}\\
\text{Subject to:}&\quad \frac{\ud X_t}{\ud t} = A_tX_t, \quad X_0 \sim p_0
\end{aligned}\label{eq:opt_prob}
\end{equation}
where $\lambda \in\mathbb{R}^{+}:=\{x\in\Re \;:\; x\geq 0\}$ is a
regularization parameter.  The limit $\lambda \downarrow 0$ is
referred to as $[\lambda=0^+]$ problem.  The symbol $\tr(\cdot)$ and
superscript $^\top$ are
used to denote matrix trace and matrix transpose, respectively.

The motivation to add the regularization is as follows: It is shown in the paper that the stochastic gradient descent (for the functional
${\sf J}$) yields the following learning algorithm for the weights $A_t$:
\begin{equation}
A^{(k+1)}_t = A_t^{(k)}+\eta_k (-\lambda A_t^{(k)} \, + \, \text{backprop update})
\label{eq:learning_lambda}
\end{equation}
for $k=1,2,\ldots$, where $\eta_k$ is the learning rate
parameter. 
Thus the parameter $\lambda$ models (small) dissipation in backprop.
In an implementation of backprop, one would expect to
obtain critical points of the $[\lambda=0^+]$ problem where the
parameter $\lambda$ is seen to provide implicit regularization.    

The contributions of this paper are as follows: The Hamilton's
formulation is introduced for the optimal control problem in \Sec{sec:Hamilton};
cf.,~\cite{le1988theoretical,farotimi1991general} for related
constructions.  The Hamilton's
equations are used to obtain a formula for the gradient of ${\sf J}$,
and subsequently derive the stochastic gradient descent learning
algorithm of the form~\eqref{eq:learning_lambda}.  The equations for
the critical points of ${\sf J}$ are obtained by applying the Maximum
Principle (Proposition~\ref{prop:opt-ham}).  Remarkably, the
Hamilton's equations for the critical points can be solved in
closed-form to obtain a complete characterization of the critical
points in terms of the solutions of a nonlinear matrix-valued
equation, referred to as the characteristic equation
(Proposition~\ref{prop:opt-vec}).  Analytical results for the normal
matrix case are described based on the use of implicit function
theorem (Theorem~\ref{cor:opt-vec}).  Numerical continuation is
employed to compute these solutions for both normal and non-normal
cases.

\section{Hamilton's formulation and the learning algorithm}
\label{sec:Hamilton}
\begin{definition}
The control Hamiltonian is the function
\begin{equation}
{\sf H}(x,y,B) = y^\top B x -\frac{\lambda}{2} \text{tr}(B^\top\,B) 
\label{eq:H}
\end{equation}
where $x\in\Re^d$ is the state, $y\in\Re^{d}$ is the
co-state, and $B\in M_d(\Re)$ is the weight matrix.
The partial derivatives are denoted as 
$\frac{\partial {\sf H}}{\partial
  x}(x,y,B) := B^\top y$, $\frac{\partial {\sf H}}{\partial
  y}(x,y,B) := B x$, and  
$\frac{\partial {\sf H}}{\partial
  B}(x,y,B) := yx^\top -\lambda B$. 
\end{definition}

Pontryagin's Maximum Principle (MP) is used to obtain the Hamilton's
equations for the optimal solutions.  The MP represents a
necessary condition satisfied by any minimizer of the optimal control
problem~\eqref{eq:opt_prob}.  
Conversely, a solution of the Hamilton's
equation is a critical point of the functional ${\sf J}$.  
The proof of the following proposition appears in the Appendix~\ref{apdx:Hamiltonian}


\begin{proposition}\label{prop:opt-ham}
Consider the terminal cost optimal control
problem~\eqref{eq:opt_prob}.  Suppose $A_t$ is the minimizer and $X_t$ is the
corresponding trajectory. Then there exists a random process
$Y:[0,T]\rightarrow \Re^d$ such that
\begin{align}
\frac{\ud X_t}{\ud t} &= +\frac{\partial {\sf H}}{\partial
  y}(X_t,Y_t,A_t)= +A_t X_t,\quad X_0\sim p_0\label{eq:Xt}\\
\frac{\ud Y_t}{\ud t}&= -\frac{\partial {\sf H}}{\partial
  x}(X_t,Y_t,A_t)=-A_t^\top Y_t,\quad Y_\tf = Z-X_\tf\label{eq:Yt}
\end{align}
and $A_t$ maximizes the expected value of the Hamiltonian
\begin{equation}\label{eq:At}
A_t = \argmax_{B\,\in\, M_d(\Re)} \;\; \E[{\sf H}(X_t,Y_t,B)]=\frac{1}{\lambda} \E[Y_t\,X_t^\top]
\end{equation}
\end{proposition}

The Hamiltonian is also used to express the first order variation in
the functional $\sfJ$. For this purpose, define 
the Hilbert space of matrix-valued functions $L^2([0,T];M_d(\Re)):=\{A:[0,T]\to
M_d(\Re) \mid 
\int_0^T \tr(A_t^\top A_t)\ud t < \infty \}$,
under the inner product $\lr{A}{V}_{L^2} :=\int_0^T
\tr(A_t^\top V_t)\ud t$. 
For any $A\in
L^2$, the gradient of the functional $\sfJ$ evaluated at $A$ is
denoted as $\nabla {\sf J}[A] \in L^2$.  It is defined using
the directional derivative formula:
\[
\lr{\nabla {\sf J}[A]}{V}_{L^2} := \lim_{\epsilon\rightarrow 0} \frac{
  {\sf J}(A+\epsilon V) -  {\sf J}(A)}{\epsilon} 
\]
where $V\in L^2$ prescribes the direction (variation) along which the derivative is
being computed.  The explicit formula for $\nabla {\sf J}$ is given by
\begin{equation}
\nabla {\sf J}[A] := - \E\left[ \frac{\partial {\sf H}}{\partial B} (X_t,Y_t,A_t)\right] = \lambda A_t - \E\left[Y_t\,X_t^\top\right]\label{eq:gradJ}
\end{equation}
where $X_t$ and $Y_t$ are the obtained by solving the Hamilton's
equations~\eqref{eq:Xt}-\eqref{eq:Yt} with the prescribed (not
necessarily optimal) weight matrix $A \in L^2$ (the details of the derivation appears in the Appendix~\ref{apdx:J-var}).  
The significance of the
formula is that the steepest descent in the objective function ${\sf
  J}$ is obtained by moving in the direction of the steepest (for
each fixed $t\in[0,T]$) ascent in the Hamiltonian ${\sf H}$.
Consequently, a stochastic gradient descent algorithm to learn the weights is: 
\begin{equation}
A_t^{(k+1)} = A_t^{(k)} - \eta_k ({\lambda A_t^{(k)} - Y_t^{(k)}\,{X_t^{(k)}}^\top}
),
\label{eq:Atl}
\end{equation}
where $\eta_k$ is the step-size at iteration $k$ and $X_t^{(k)}$ and $Y_t^{(k)}$ are obtained by solving the Hamilton's equations~\eqref{eq:Xt}-\eqref{eq:Yt}:
\begin{align}
\text{(Forward propagation)}\quad \frac{\ud }{\ud t} X_t^{(k)}&= +A_t^{(k)} X_t^{(k)},\quad \text{with
  init. cond.}\;
\; X_0^{(k)}\label{eq:Xtl}\\
\text{(Backward propagation)}\quad \frac{\ud }{\ud t} Y_t^{(k)} &=-A_t^{(k)\top} Y_t^{(k)},\quad Y_\tf^{(k)} =
\underbrace{Z^{(k)} -X_\tf^{(k)}}_{\text{error}} 
\label{eq:Ytl}
\end{align}
based on the sample input-output $(X^{(k)},Z^{(k)})$. 
Note the forward-backward structure of the algorithm: In the forward pass, the network output $X_\tf^{(k)}$ is obtained given the input $X^{(k)}_0$; In the backward pass, the error between the network output $X^{(k)}_\tf$ and true output $Z^{(k)}$ is computed and propagated backwards.
By setting $\lambda=0$ the standard backprop algorithm is obtained.  A
convergence result for the learning algorithm for the $[\lambda=0]$ case  appears in the Appendix~\ref{apdx:convergence}.

In the remainder of this paper, the focus is on the analysis of the
critical points.

\section{Critical points}
\label{sec:critical_points}

For continuous-time networks, the critical points of the $[\lambda=0]$
problem are all global minimizers (An
analogous result for residual MNN appears in~\cite[Theorem
2.3]{DBLP:journals/corr/HardtM16}).

\begin{theorem}
Consider the $[\lambda=0]$ optimization
problem~\eqref{eq:opt_prob} with non-singular $\Sigma$.  
For this
problem (provided a minimizer exists) every critical point is a global minimizer.  That is,    
\begin{equation*}
\quad\nabla {\sf J}[A] = 0 \quad \Longleftrightarrow \quad  {\sf J}(A)= {\sf J}^* := \min_{A}\sfJ[A]
\end{equation*}  
Moreover, for any given (not necessarily optimal) $A\in L^2$,
\begin{equation}
\|\nabla  {\sf J}[A]\|_{L^2}^2 \geq \, \tf \, e^{-2\int_0^\tf
  \sqrt{\tr(A_t^\top A_t)} \; \ud t}\;
\lambda_{\text{min}}(\Sigma)( {\sf J}(A)- {\sf J}^*)  
\label{eq:gradJ-J}
\end{equation}
\label{prop:critical-pt}
\end{theorem}
\begin{proof} (Sketch) For the linear system~\eqref{eq:lin_net}, the
  fundamental solution matrix is denoted as $\phi_{t;t_0}$.  The
  solutions of the Hamilton's equations~\eqref{eq:Xt}-\eqref{eq:Yt} are
  given by
\begin{align*}
X_t &= \phi_{t;0} X_0,\quad
Y_t = \phi_{T;t}^\top(Z-X_\tf)
\end{align*}
Using the formula~\eqref{eq:gradJ} upon taking an expectation
\begin{align*}
\nabla \sfJ[A] = -\phi_{T;t}^{\top}(R-\phi_{T;0})\Sigma \phi_{t;0}^\top
\end{align*}
which (because $\phi$ is invertible) proves that: 
\begin{equation*}
\quad\nabla {\sf J}[A] = 0 \quad \Longleftrightarrow \quad \phi_{T;0}=R\quad \Longleftrightarrow \quad  {\sf J}(A)= {\sf J}^* := \min_{A}\sfJ[A]
\end{equation*}  
The derivation of the bound~\eqref{eq:gradJ-J} is equally straightforward
and appears in the Appendix~\ref{apdx:critical-pt}. 
\end{proof} 
Although the result is attractive, the conclusion is somewhat
misleading because (as we will demonstrate with examples) even a small
amount of regularization can lead to local (but not global) minimum as
well as saddle point solutions.

\noindent {\bf Assumption:} 
 The following assumption is made throughout the remainder of this paper:
\begin{romannum}
\item {\bf Property P1:}  The matrix $R$ has no eigenvalues on
  $\mathbb{R}^{-}:=\{x\in\Re \;:\; x\le 0\}$.  The matrix $R$ is
  non-derogatory.  That is, no eigenvalue of $R$ appears in more than
  one Jordan block.  
\end{romannum}

For the scalar ($d=1$) case, this property means $R$ is
strictly positive.  For the scalar case, $\phi_{\tf,0}=e^{\int_0^\tf
  A_t \ud t}$ and the positivity of $R$ is seen to be necessary to obtain a
meaningful approximation.  

For the vector case, this property represents a sufficient condition
such that $\log(R)$ can be defined as a real-valued matrix.  That is,
under property (P1), there exists a (not necessarily unique\footnote{Under Property (P1),
$\log (R)$ is uniquely defined if and only if all the eigenvalues of
$R$ are positive.  When not unique there are countably many
matrix logarithms, all denoted as $\log(R)$. 
The principal logarithm
of $R$ is the unique such matrix whose eigenvalues lie in the strip
$\{z\in\mathbb{C}\;:\; -\pi< \text{Im}(z) < \pi\}$.}) matrix
$\log(R)\in M_d(\Re)$ whose matrix exponential $e^{\log(R)}=R$;
cf.,~\cite{culver1966existence,higham2014functions}.  The logarithm is
trivially a minimum for the
$[\lambda=0]$ problem.  Indeed, $A_t \equiv \frac{1}{\tf} \log(R)$
gives $X_t = e^{\frac{\log(R)}{\tf} t} X_0$ and thus $X_T = e^{\log(R)} X_0 = R
X_0$.  This shows $A_t$ can be made arbitrarily small by
choosing a large enough depth $\tf$ of the network.  An analogous
result for the linear residual MNN appears in~\cite[Theorem
2.1]{DBLP:journals/corr/HardtM16}.  The question then is whether the
constant solution $A_t \equiv \frac{1}{\tf} \log(R)$ is also obtained
as a critical point for the $[\lambda=0^+]$ problem?


The following proposition provides a characterization of the
critical points (for the general $\lambda\in \Re^+$ problem) in terms of
the solutions of a matrix-valued characteristic equation: 


\begin{proposition}\label{prop:opt-vec}
The general solution of the Hamilton's
equations~\eqref{eq:Xt}-\eqref{eq:At} is given by
\begin{align}
X_t &= e^{2t \Omega}\; e^{t \A0^\top} X_0\label{eq:XtC}\\
Y_t &= e^{2t \Omega} \; e^{(\tf-t) \A0} \; e^{-2\tf \Omega}\, (Z-X_\tf)
\label{eq:YtC}\\
A_t &= e^{2t\Omega}\A0 e^{-2t \Omega}  
\label{eq:opt-At-vec}
\end{align}
where $\A0\in M_d(\Re)$ is an arbitrary solution of the characteristic equation
\begin{equation}
\lambda \A0 = F^\top (R-F)\Sigma
\label{eq:char-vec-A}
\end{equation}
where $F:=e^{2\tf\Omega} \, e^{\tf \A0^\top}$ and the matrix $\Omega:=\frac{1}{2}(\A0-\A0^\top)$ is
the skew-symmetric component of $\A0$.  
The associated cost is
given by 
\[
{\sf J}[A] = \frac{\lambda\tf}{2}\tr\left(\A0^\top\A0 \right)+\frac{1}{2}\tr\left((F-R)^\top(F-R)\Sigma\right) + \frac{1}{2}\expect{|\noise|^2}
\] 
And the following holds:
\begin{equation*}
A_t \equiv \A0 ~~\underset{\text{(i)}}{\Longleftrightarrow}~~ \A0 ~~\text{is normal} ~~\overset{(\Sigma=I)}{\underset{\text{(ii)}}{\Longrightarrow}}~~R~~\text{is normal}
\end{equation*}

%
\end{proposition}

\begin{proof}(Sketch)
Differentiating both sides of~\eqref{eq:At} with respect to $t$ and using the Hamilton's
equations~\eqref{eq:Xt}-\eqref{eq:Yt}, one obtains
\begin{equation}
\frac{\ud A_t}{\ud t} = -A_t^\top A_t + A_tA_t^\top \label{eq:Atdt}
\end{equation}  
whose general solution is given by~\eqref{eq:opt-At-vec}.  It is
easily verified that, with $A_t$ given
by~\eqref{eq:opt-At-vec}, formulae~\eqref{eq:XtC}-\eqref{eq:YtC} are solutions
of equations~\eqref{eq:Xt}-\eqref{eq:Yt}. The characteristic equation is obtained by using the formula~\eqref{eq:At}. The claim (i) easily follows from \eqref{eq:Atdt}. 
To prove (ii): if $\A0$ is normal, then $\A0$ and $\Omega$ commute, therefore $F =e^{\tf \A0}$. Hence the characteristic equation simplifies to
$
\lambda \A0 = e^{\A0^\top}(R-e^{\A0})\Sigma
$ or equivalently $\lambda \A0e^{-\A0^\top}\Sigma^{-1} +e^{\A0} = R$. Therefore if $\A0$ and $\Sigma$ commute (always true when $\Sigma=I$), $R$ is a normal matrix (the detailed proof appears in the Appendix~\ref{apdx:opt-vec}).
\end{proof}   

\begin{remark}
The result shows that the answer to the question posed above concerning the
constant solution $A_t \equiv \frac{1}{\tf} \log(R)$ is false in
general for the $[\lambda=0^+]$ problem: For $\lambda>0$ and 
$\Sigma_0=I$, a constant solution is a critical point only if $R$ is a
normal matrix.  For the generic case of non-normal $R$, any critical point is necessarily
non-constant for any positive choice of the parameter $\lambda$.  Some
of these non-constant critical points are described as part of the
Example~\ref{ex:non-normal}.
\end{remark}

The above proposition is useful because it helps reduce the
infinite-dimensional problem to a finite-dimensional characteristic
equation~\eqref{eq:char-vec-A}.  The solutions $\A0$ of the
characteristic equation fully parametrize the solutions of the
Hamilton's equations~\eqref{eq:Xt}-\eqref{eq:At} which in turn represent the
critical points of the optimal control problem~\eqref{eq:opt_prob}.   

The matrix-valued nonlinear characteristic
equation~\eqref{eq:char-vec-A} is still formidable. To gain analytical and numerical insight into the matrix case, the following
strategy is employed:
\begin{romannum}
\item A solution $\A0$ is obtained by setting $\lambda=0$ in the
  characteristic equation.  The corresponding equation is
\[
e^{\tf(\A0-\A0^\top)} e^{\tf \A0^\top} = R
\]
This solution is denoted as $\A0(0)$.  
\item Implicit function theorem is used to establish
  (local) existence of a solution branch $\A0(\lambda)$ in a
  neighborhood of $\lambda=0$ solution.  
\item Numerical continuation is used to compute the solution $\A0(\lambda)$ 
  as a function of the parameter $\lambda$. 
\end{romannum}

The following theorem provides a characterization of normal solutions
$\A0$ for the case where $R$ is assumed to be a normal matrix and
$\Sigma=I$. 

\begin{theorem}\label{cor:opt-vec}
Consider the characteristic equation~\eqref{eq:char-vec-A} where $R$
is assumed to be a normal matrix that satisfies the Property (P1) and
$\Sigma=I$.   
\begin{romannum}
\item For $\lambda=0$ the normal solution of~\eqref{eq:char-vec-A} is
  given by $ \frac{1}{T}
  \log (R)$. 
\item For each such solution, there exists a neighborhood ${\cal N}\subset
  \Re^+$ of $\lambda=0$
  such that the solution of the characteristic
  equation~\eqref{eq:char-vec-A}  is well-defined as a continuous map
  from $\lambda\in {\cal N} \rightarrow \A0(\lambda) \in M_d(\Re)$ with $\A0(0)=\frac{1}{T} \log (R)$. This
  solution is given by the asymptotic formula
\[
\A0(\lambda) = \frac{1}{T} \log (R) - \frac{\lambda}{T^2} (RR^\top)^{-1}\log(R) + O(\lambda^2) 
\]
\end{romannum} 
\end{theorem}
\begin{proof}(sketch)
(i) If $\A0$ is normal, $e^{\tf(\A0-\A0^\top)}e^{\tf \A0}=e^{\tf
  \A0}$. Hence, for $\lambda=0$, the characteristic equation becomes
$e^{\tf \A0}=R$ whose solution is $\A0=\frac{1}{\tf}\log(R)$.  
(ii) For $\lambda>0$ and $\Sigma=I$ the characteristic equation is
expressed as
$\lambda \A0e^{-\A0^\top} +e^{\A0} = R$.
If $\A0$ is a normal solution, the equation can be
diagonalized with a unitary transformation, which can be used to compute the solution.  
The details appear in the Appendix~\ref{apdx:cor-opt-vec}.
\end{proof}

\begin{remark}
For the scalar case $\log(\cdot)$ is single-valued function. Therefore, $A_t\equiv\A0=\frac{1}{T}\log(R)$ is the unique critical point (minimizer) for the $[\lambda=0^+]$ problem. While the $[\lambda=0^+]$ problem admits a unique minimizer, the
$[\lambda=0]$ problem does not.  In fact, any
$A_t$ of the form
$A_t =  \frac{1}{T}\log(R) + \tilde{A}_t$
where $\int_0^{\tf} \tilde{A}_t \ud t = 0$ is also a minimizer of the
$[\lambda=0]$ problem.  So, while there are infinitely many minimizers
of the $[\lambda=0]$ problem, only one of these survives with even a
small amount of regularization.  
A global characterization of critical points as a function of parameters
$(\lambda,R,\Sigma_0,T)\in \Re^+\times \Re^+ \times \Re^+\times \Re^+$
is possible and appears in the Appendix~\ref{apdx:scalar-case}.  \label{rem:scalar}
\end{remark}
In the following, numerical solutions for two example problems are
described.  The second example involves a non-normal matrix $R$, and
as such is not covered by Theorem~\ref{cor:opt-vec}.  

\begin{example}[Normal]\label{ex:normal}
Consider the characteristic equation \eqref{eq:char-vec-A} with $R
= \begin{bmatrix} 0 & -1 \\ 1 & 0 \end{bmatrix}$ (rotation in the
plane by $\pi/2$), $\Sigma = I$ and $T=1$. 
For $\lambda = 0$, the normal solutions of the characteristic equation
are given by multivalued matrix logarithm function:$$\log(R) =
(\pi/2+2n\pi)\begin{bmatrix} 0 & -1 \\ 1 & 0 \end{bmatrix}
=:\A0(0;n),\quad n=0,\pm1,\pm 2,\hdots$$ It is easy to verify that
$e^{\A0(0;n)} = R$. 
$\A0(0;0)$ is referred to as the principal logarithm.
	\begin{figure}[t]
		\centering
		\includegraphics[width=0.49\columnwidth]{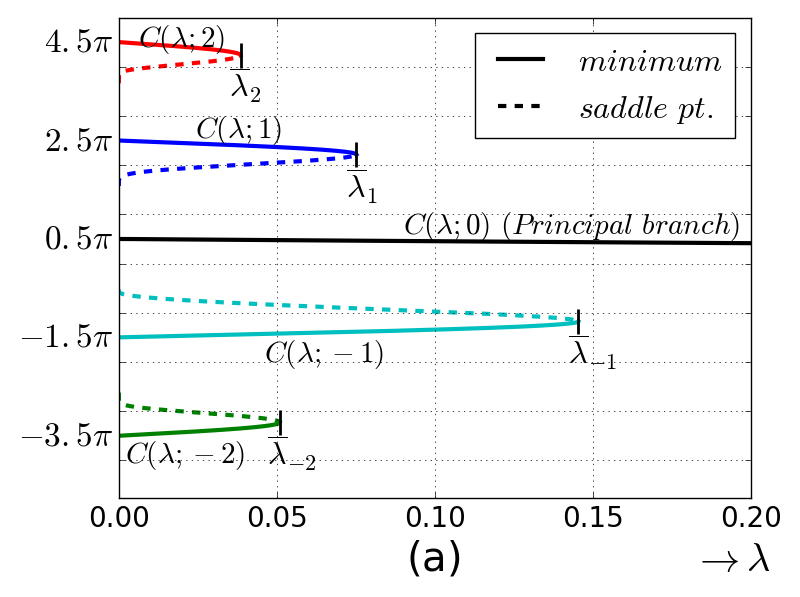}
		\includegraphics[width=0.49\columnwidth]{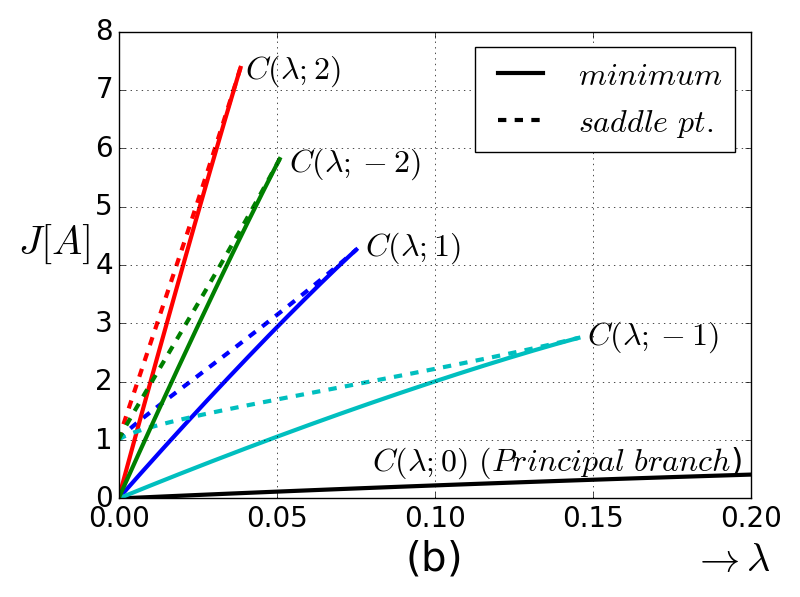}
\vspace{-0.1in}
		\caption{(a) Critical points in Example~1 (the $(2,1)$
                  entry of the solution matrix $C(\lambda;n)$ is
                  depicted for $n=0,\pm1,\pm2$); (b) The cost $J[A]$
                  for these solutions.}\label{fig:ex1}
\vspace{-0.1in}
	\end{figure}

	The software package PyDSTool~\cite{pydstool07} is used to numerically continue the solution $C(\lambda;n)$ as a function of the parameter $\lambda$. Figure \ref{fig:ex1}(a) depicts the solutions branches in terms of the $(2,1)$ entry of the matrix $C(\lambda;n)$ for $n=0,\pm1,\pm 2$. The following observations are made concerning these solutions:
	\begin{romannum}
		\item For each fixed $n\neq0$, there exist a range
                  $(0,\bar{\lambda}_n)$ for which there exist two
                  solutions, a local minimum and a saddle point. At
                  the limit (turning) point $\lambda =
                  \bar{\lambda}_n$, there is a qualitative change in
                  the solution from a minimum to a saddle point.
		\item As a function of $n$, $\bar{\lambda}_n$ decreases monotonically as $|n|$ increases. 
		For $\lambda > \bar{\lambda}_{-1}$, only a single solution, the principal branch $C(\lambda;0)$ was found using numerical continuation.
		\item Along the branch with a fixed $n\neq0$, as
                  $\lambda \downarrow 0$, the saddle point solution
                  escapes to infinity. That is as $\lambda \downarrow
                  0$, the saddle point solution
                  $C(\lambda;n)\to(\pi/2+(2n-1)\pi)\begin{bmatrix}-\infty
                    & -1\\1 & -\infty\end{bmatrix}$. The associated
                  cost $J[A]\downarrow 1$ (The cost of global minimizer $J^*=0$).
		\item Among the numerically obtained solution
                  branches, the principal branch $C(\lambda;0)$ has
                  the lowest cost. Figure~\ref{fig:ex1} (b) depicts
                  the cost for the solutions depicted in
                  Figure~\ref{fig:ex1} (a).	
	\end{romannum}

The numerical calculations indicate that while the $[\lambda=0]$ problem has infinitely
many critical points (all global minimizers), only a few of these
critical points persist for any finite positive value of $\lambda$.
Moreover, there exists both local (but not global) minimum as well as
saddle points for this case.  Among the solutions computed, the
principal branch (continued from the principal logarithm $C(0;0)$) has
the minimum cost. 
	
	
\end{example}

\begin{example}[Non-normal]\label{ex:non-normal}
Numerical continuation is used to obtain solutions for non-normal $R = \begin{bmatrix} 0 & -1 \\ 1 & \mu \end{bmatrix}$, where $\mu$ is a continuation parameter and $T=1$.
Figure~\ref{fig:ex2}(a) depicts a solution branch as a function of parameter $\mu$. 
The solution is initialized with the normal solution $\A0(0;0)$ described in Example~1. 
By varying $\mu$, the solution is continued to $\mu = \pi/2$
(indicated as ${\color{red}\diamond}$ in part~(a)). This way, the solution $\A0=\begin{bmatrix} 0 & 0 \\ \frac{\pi}{2} & 0 \end{bmatrix}$ is found for $R = \begin{bmatrix} 0 & -1 \\ 1 & \frac{\pi}{2} \end{bmatrix}$. 
It is easy to verify that $\A0$ is a solution of the characteristic
equation \eqref{eq:char-vec-A} for $\lambda = 0$ and $T=1$. For this
solution, the critical point of the optimal control problem
\begin{equation*}
A_t = \begin{bmatrix} -\pi \sin(\pi t) & \pi\cos(\pi t) - \pi\\
\pi\cos(\pi t) + \pi & \pi \sin(\pi t)
\end{bmatrix}
\end{equation*}
is non-constant. 
It is noted that the principal logarithm $\log(R) = \begin{bmatrix} -\gamma\tan \gamma &    -\gamma\sec \gamma\\ \gamma\sec \gamma&  \gamma\tan \gamma \end{bmatrix}$, where $\gamma = \sin^{-1}\left(\dfrac{\pi}{4}\right)$. The regularization cost for the non-constant solution $A_t$ is strictly smaller than the constant $\frac{1}{T}\log(R)$ solution:
	\begin{align*}
	\int_0^1\tr(A_tA_t^\top) \ud t = \int_0^1 \tr(\A0\A0^\top)\ud t &= \frac{\pi^2}{4} < 3.76 = 
	\int_0^1 \tr(\log(R)\log(R)^\top) \ud t
	\end{align*}
	

Next, the parameter $\mu=\frac{\pi}{2}$ is fixed, and the solution
continued in the parameter $\lambda$.  Figure~\ref{fig:ex2}(b) depicts
the cost $J[A]$ for the resulting solution branch of critical points
(minimum).  The cost with the constant $\frac{1}{T}\log(R)$ is also
depicted. It is noted that the
    latter is not a critical point of the optimal control problem for
    any positive value of $\lambda$.  

	\begin{figure}[t]
		\centering
		\includegraphics[width=0.49\columnwidth]{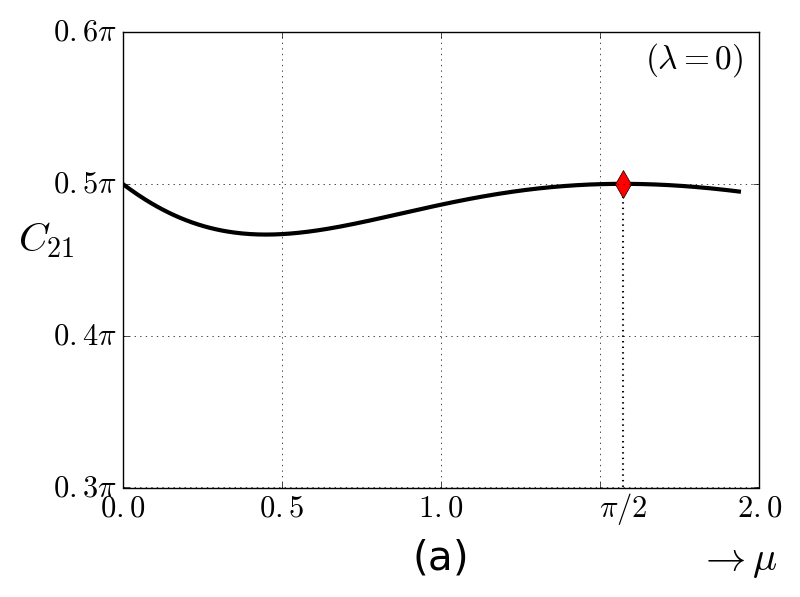}
		\includegraphics[width=0.49\columnwidth]{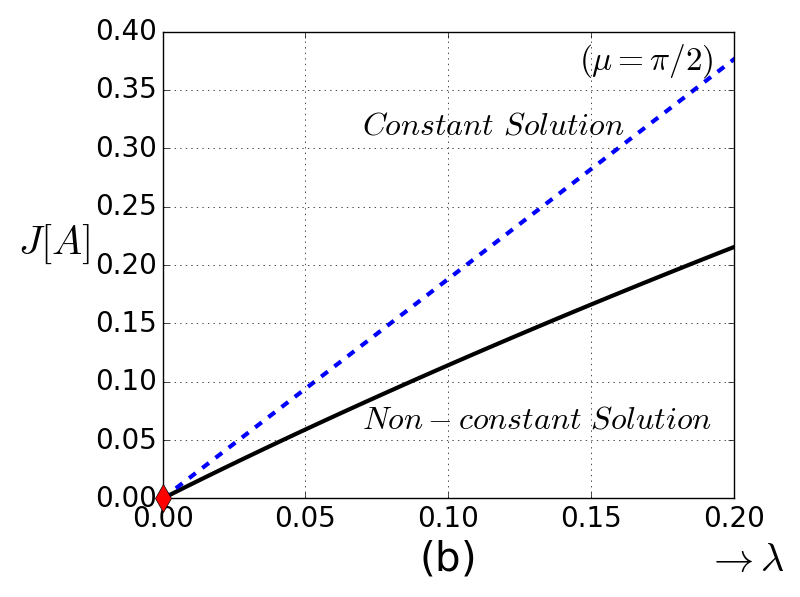}
\vspace{-0.1in}
		\caption{(a) Numerical continuation of the solution in
                  Example~2; (b) The cost $J[A]$ for the critical
                  point (minimum) and the constant
                  $\frac{1}{T}\log(R)$ solution.} \label{fig:ex2} 
\vspace{-0.1in}
	\end{figure}
	
\end{example}


\section{Conclusion}
In this paper, the non convex optimization problem of learning the weights of a linear network with a regularized model of mean-squared loss function
was introduced and studied. 
The regularized model~\eqref{eq:opt_prob} is likely to
reveal features (both good and bad) which are {\em robust} and as such
likely to be seen in an implementation of the backprop algorithm.
For example, it was shown that the regularization serves to constrain the number and type of critical points (see Remark~\ref{rem:scalar}). Also, saddle points can appear 
when none exist for the $[\lambda=0]$ problem (see Example~\ref{ex:normal}).
 The focus of the continuing research concerns the generalization property and the stability of the critical points. 

\bibliographystyle{plainnat}
\bibliography{neural}

\section{Appendix}\label{sec:appdx}
\newP{Notation} For all $B \in M_d(\Re)$, the Frobenius norm is denoted as $\|B\|$ given by $\|B\|:=\sqrt{\tr(BB^\top)}$.
\subsection{Scalar case} \label{apdx:scalar-case}

The scalar case is proved using elementary means and is useful to
both introduce the characteristic equation as well as highlight the difference
between the $[\lambda=0]$ and the $[\lambda=0^+]$ problems.

\begin{theorem}\label{prop:opt_scalar}
Consider the terminal cost optimal control problem~\eqref{eq:opt_prob}
for the scalar ($d=1$) case with $R > 0$ and $\Sigma_0 = \E [X_0^2] >
0$ given.  If $A_t$ is a minimizer then
\begin{equation}
A_t \equiv \A0, \quad
X_t = e^{t \A0} X_0\label{eq:crit_scalar}
\end{equation}
where the constant $\A0$ is a solution of the characteristic equation
\begin{equation}
\lambda \A0=  e^{ \tf \A0}(R-e^{ \tf \A0})  \Sigma_0
\label{eq:char-A}
\end{equation} 
Conversely a solution $\A0$ of the characteristic
equation~\eqref{eq:char-A} defines a critical point~\eqref{eq:crit_scalar} of the optimal control
problem~\eqref{eq:opt_prob}.  

The following is a complete characterization of the solutions $\A0$ of the
characteristic equation~\eqref{eq:char-A} as a function of parameters
$(\lambda,R,\Sigma_0,T)\in \Re^+\times \Re^+ \times \Re^+\times \Re^+$:
\begin{romannum}
\item For $\lambda\in[0,2e^3\Sigma_0 T]$ there exists a 
  unique solution.  The associated solution obtained
  using~\eqref{eq:crit_scalar} is a minimizer.    
\item In the asymptotic limit as $\lambda \downarrow 0$, the minimizer is
given by an asymptotic expansion
\begin{equation}
\A0 = \frac{1}{\tf} \log(R) -  \lambda \frac{\log(R)}{\tf^2 R^2 \Sigma_0} 
+ O(\lambda^2)\label{eq:C-asymptotic}
\end{equation}
The unique solution for the $\lambda=0^+$ problem, obtained by
  retaining the first order term, is given by $\A0 = \frac{1}{\tf}
  \log(R)$.
\item For $\lambda>2e^3\Sigma_0 T$, there exists an interval such that for $R \in
 [R_1(\lambda),R_2(\lambda)]$ there are exactly 3
 solutions of the characteristic equation.  For $R>R_2(\lambda)$ or
 $R<R_1(\lambda)$ there exists exactly one solution. 
\end{romannum}
\end{theorem}
\begin{proof}
In the scalar case, the state is given by the explicit formula $X_t=e^{\int_0^tA_s\ud s}X_0$. Therefore, the objective function
\begin{align*}
\sfJ [A]= \lambda\int_0^\tf A_t^2\ud t +(e^{\int_0^\tf A_t\ud t}-R)^2 \Sigma_0
\end{align*}
Using the Jensen's inequality 
\begin{align*}
\sfJ[A] &\geq \frac{\lambda}{\tf}(\int_0^\tf A_t \ud t)^2 + (e^{\int_0^\tf A_t\ud t}-R)^2 \Sigma 
\end{align*}
with an equality iff $A_t \equiv \A0$, a constant. Therefore
\begin{align*}
\min_{A\in L^2}~J(A) 
= \min_{\A0 \in \Re} ~\lambda \tf \A0^2 +(e^{\tf \A0}-R)^2 \Sigma
\end{align*}
The characteristic equation is the first order optimality condition of the right hand side.
 
\begin{romannum}
\item Denote $\tlambda=\frac{\lambda}{\tf \Sigma}$ and $\tA0=\tf \A0$ to write the characteristic equation as
\begin{equation}
f(\tA0,\tlambda) := \tlambda \tA0 e^{-\tA0} + e^{\tA0} = R \label{eq:f}
\end{equation}
For $\tlambda = 0$, the solution $\tA0=\log(R)$.
For $\lambda > 0$, $f$ is onto (since $f$ is continuous and $\lim_{\tA0 \to\pm\infty}f(\tA0) = \pm\infty$). 
Therefore, there exists at least one solution for each given $R$.
Since $f'(\tA0) = \lambda e^{-\tA0}(1-\tA0)+e^{\tA0} > 0$ for $\tA0 \leq 1$, $f$ is monotone on $(-\infty,1]$. 
Also $f'(\tA0) = 0 \Leftrightarrow \lambda = \frac{e^{2\tA0}}{\tA0-1}$ and $\frac{e^{2\tA0}}{\tA0-1}$ is a unimodal convex function for $\tA0 > 1$ with minimum $2e^3$ at $\tA0 = 3/2$. Therefore for $\tlambda \leq 2e^3$, $f$ is monotone over entire $\mathbb{R}$. This implies that the solution to $f(\tA0) = R$ is unique for $\tlambda\leq 2e^3$.
\item At $\tlambda=0$, $\tA0=\log(R)$. Also $f'(\log(R),0)=R \neq 0$. 
So by the implicit function theorem there exists a unique solution $
\tlambda \to \tA0(\tlambda)$ in a neighborhood of $0$. The asymptotic formula~\eqref{eq:C-asymptotic} for the solution is obtained by substituting regular perturbation expansion 
$\tA0= \tA0_0 + \lambda \tA0_1 + O(\lambda^2)$ into \eqref{eq:f}.
 \begin{equation*}
f(\tA0) = \lambda\tA0_0 e^{-\tA0_0} +e^{\tA0_0}(1+\lambda\tA0_1) + O(\lambda^2)=R
 \end{equation*}
Collecting the zeroth and the first order terms, one obtains
$\tA0_0 = \log(R)$ and $\tA0_1 =-\frac{\log(R)}{R^2}$.
 
\item If $\tlambda > 2e^3$, $f'(\tA0) = 0$ has two solutions, $\tA0_1 \in (3/2,\infty)$ and $\tA0_2 \in (1,3/2)$. Therefore for $R \in [f(\tA0_1), f(\tA0_2)]$, $f(\tA0) = R$ has three solutions.
\end{romannum}
\end{proof}


\subsection{Proof of the Proposition~\ref{prop:opt-ham} (Hamiltonian formulation)} \label{apdx:Hamiltonian}
Let $A_t$ be the minimizer of~\eqref{eq:opt_prob}. Define $X_t$ and $Y_t$ as the solutions of the Hamilton's equations~\eqref{eq:Xt}-\eqref{eq:Yt}. We show $A_t$ satisfies~\eqref{eq:At} as follows:
For $s \in[0,T]$ and $B \in M_d(\Re)$ consider a (needle) variation of the form:
\begin{equation*}
A^{(\epsilon)}_t = \begin{cases}
B\quad &t \in [s-\epsilon,s]\\
A_t\quad &t \notin [s-\epsilon,s] 
\end{cases}
\end{equation*}
Let $X_t^{(\epsilon)}$ denote the solution to the Hamitonian equation-\eqref{eq:Xt} with $A_t^{(\epsilon)}$. 
It is given by:
\begin{equation*}
X^{(\epsilon)}_t = X_t + \epsilon \tX_t + O(\epsilon^2)
\end{equation*}
where for $t<s$, $\tX_t = 0$ and for $t>s$, $\tX_t$ is the solution of 
\begin{equation*}
\frac{\ud \tX_t}{\ud t} = A_t \tX_t, \quad \text{with i.c}\quad \tX_s =  (B-A_s)X_s
\end{equation*}
The perturbed cost is 
\begin{align*}
\sfJ[A^{(\epsilon)}]  = \sfJ[A] &
+  \epsilon \lambda (\text{Tr}(B^\top B)-\text{Tr}(A_s^\top A_s) )+ 2\epsilon \expect{(X_\tf-Z)^\top \tX_\tf}   + O(\epsilon^2)
\end{align*}    
Since $A_t$ is a minimizer
\begin{equation*}
\frac{\lambda}{2} (\text{Tr}(B^\top B)-\text{Tr}(A_s^\top A_s) ) +  \expect{(X_\tf-Z)^\top\tX_\tf} \geq 0 
\end{equation*}
The next step is to obtain $(X_\tf-Z)^\top \tX_\tf$ in terms of $Y_t$. 
By construction $\frac{\ud}{\ud t} Y_t^\top \tX_t=0$. Therefore,
\begin{equation*}
(Z-X_\tf)^\top \tX_\tf = Y_\tf^\top \tX_\tf = Y_s^\top \tX_s = Y_s^\top (B-A_s) X_s
\end{equation*}
and hence
\begin{equation*}
\frac{\lambda}{2} (\text{Tr}(B^\top B)-\text{Tr}(A_s^\top A_s) ) -  \expect{Y_s^\top (B-A_s)X_s} \geq 0,
\end{equation*}
On collecting the terms, one obtains
\begin{equation*}
\expect{ H(X_s,Y_s,A_s)} \geq \expect{H(X_s,Y_s,B)}\quad \forall B \in M_d(\Re)
\end{equation*}
Since $s\in [0,T]$ is arbitrary, the result follows.

\subsection{ First order variation of $\sfJ$}\label{apdx:J-var}
Let $X_t$ and $Y_t$ be the solutions to the Hamilton's equations-\eqref{eq:Xt}-\eqref{eq:Yt} with weight matrix $A_t$.
Define $A^{(\epsilon)}_t = A_t + \epsilon V_t$.   Let $X^{(\epsilon)}_t$ and $Y^{(\epsilon)}_t$ be the solutions to the Hamilton's equations with weight matrix $A^{(\epsilon)}_t$. 
In the limit as $\epsilon \to 0$, $X_t^{(\epsilon)}$ is given by the asymptotic formula $X_t^{(\epsilon)} =  X_t + \epsilon \tX_t + O(\epsilon^2)$ where
\begin{equation*}
\frac{\ud \tX_t}{\ud t}  = A_t \tX_t  + V_t X_t,\quad \tX_0=0
\end{equation*}
In terms of $\tX_t$, the objective function 
\begin{equation*}
\sfJ[A^{(\epsilon)}] = \sfJ[A] + \epsilon \left(\lambda \int_0^\tf \tr(A_t^\top V_t) + \E[{(X_T-Z)^\top \tX_T}]\right) + O(\epsilon^2)
\end{equation*}
Use the definition of $Y_t$ to express $(X_T-Z)^\top \tX_T$ as 
\begin{align*}
(Z-X_T)^\top \tX_T = Y_T^\top \tX_T= \int_0^T \frac{\ud}{\ud t} (Y_t^\top \tX_t)\ud t = \int_0^T(-Y_t^\top A_t\tX_t + Y_t^\top A_t\tX_t + Y_t^\top V_tX_t)\ud t = \int_0^T Y_t^\top V_tX_t \ud t
\end{align*}
Therefore,
\begin{align*}
\sfJ[A^{(\epsilon)}] &=\sfJ[A] + \epsilon \int_0^T \E{\left[ \tr\left((\lambda A_t^\top  - X_tY_t^\top) V_t\right)\right]}\ud t + O(\epsilon^2) 
\end{align*}
On the other hand $\frac{\partial H}{\partial B} (x,y,B)=\lambda B + yx^\top$. Therefore 
\begin{align*}
\sfJ[A^{(\epsilon)}] - \sfJ[A]= \epsilon \int_0^T \tr\left(\E\left[\frac{\partial H}{\partial B} (X_t,Y_t,A_t)\right]^\top V_t\right)\ud t + O(\epsilon^2) 
\end{align*}
which gives the result $\nabla \sfJ[A]=-\E[\frac{\partial H}{\partial B} (X_t,Y_t,A_t)]$. 


\subsection{Proof of the Theorem~\ref{prop:critical-pt}} \label{apdx:critical-pt}
For the [$\lambda=0$] problem, the gradient $\nabla \sfJ[A]$ is (by~\eqref{eq:gradJ})
\begin{align*}
\nabla \sfJ[A] = -\E[Y_tX_t^\top]
\end{align*}
where $X_t$ and $Y_t$ solve the Hamilton's equations-\eqref{eq:Xt}-\eqref{eq:Yt}. Define the state transition matrix $\phi(t,t_0)$ of the differential equation $\frac{\ud X_t}{\ud t}= A_tX_t$ according to:
\begin{equation*}
\frac{\ud \phi({t,t_0})}{\ud t} = A_t \phi({t,t_0}),\quad \phi({t_0,t_0})=I
\end{equation*}
In terms of the transition matrix,
\begin{align*}
X_t &= \phi(t,0) X_0,\quad
Y_t = \phi({T,t})^\top(Z-X_\tf)
\end{align*}
Therefore
\begin{align*}
\nabla \sfJ[A] = -\phi({T,t})^{\top}(R-\phi({T,0}))\Sigma \phi({t,0})^\top=:\psi_t
\end{align*}
Since $\phi({t,t_0})$ is invertible 
\begin{equation*}
\nabla \sfJ[A]= 0 \quad \Leftrightarrow \quad R=\phi({\tf,0}) \quad \Leftrightarrow \quad \sfJ[A]=J^* = \E[|W|^2]
\end{equation*}
For each fixed $t \in [0,T]$
\begin{align*}
\|\psi_t\|^ 2 &= \|\phi({T,t})^\top(R-\phi({T,0}))\Sigma\phi({t,0})^\top \|^2\\ &\geq \lambda_{\text{min}} (\phi({T,t})^\top\phi({T,t}))\lambda_{\text{min}} (\phi({t,0})^\top\phi({t,0})) \|(R-\phi({T,0}))\Sigma \|^2 \\
&\geq \lambda_{\text{min}} (\phi({T,t})^\top\phi({T,t}))\lambda_{\text{min}} (\phi({t,0})^\top\phi({t,0})) \lambda_{\text{min}} (\Sigma) \tr((R-\phi({T,0}))^\top(R-\phi({T,0}))\Sigma) \\
&\geq e^{-2\int_0^T|\|A_t\|\ud t} \lambda_{\text{min}} (\Sigma) (\sfJ[A]-\sfJ^*)
\end{align*}
where we used Lemma~\ref{lem:ineq} (see below) in the last step. Integrating the inequality on $[0,T]$ yields the result.

\begin{lemma}\label{lem:ineq}
Let $\phi({t,t_0})$ be the state transition matrix defined according to $\frac{\ud}{\ud t} \phi({t,t_0}) = A_t \phi({t,t_0})$ with $\phi({t_0,t_0})=I$. Then,
\begin{equation*}
 e^{-2\int_0^t \|A_t\| \ud t} \leq \lambda_{\text{min}} (\phi({t,0})^\top\phi({t,0}))\leq \lambda_{\text{max}} (\phi({t,0})^\top\phi({t,0})) \leq e^{2\int_0^t \|A_t\| \ud t} 
\end{equation*}
%
\end{lemma}
\begin{proof}
Observe that
\begin{align*}
\lambda_{\text{max}} (\phi({t,0})^\top\phi({t,0})) &= \max_{x \neq 0} \frac{x^\top\phi_{t,0}^\top\phi_{t,0}x}{x^\top x}=\max_{x_0 \neq 0} \frac{x_t^\top x_t}{x_0^\top x_0}
\end{align*}
Now, 
\begin{align*}
\frac{\ud}{\ud t} |x_t|^2 = x_t^{\top}(A_t+A_t^\top)x_t \leq \lambda_{\text{max}}(A_t+A_t^\top)|x_t|^2 \leq 2\|A_t\| |x_t|^2
\end{align*}
where the last inequality follows because $2\|A_t\|=\|A_t+A_t^\top\|  \geq \lambda_{\text{max}}(A_t+A_t^\top)$. Therefore,
$|x_t|^2 \leq e^{2\int_0^t \|A_t\| \ud t} |x_0|^2
$
which gives the upper bound. The calculation for the lower bound is similar.
\end{proof}


\subsection{Convergence of the learning algorithm} \label{apdx:convergence}
\begin{proposition}
Consider the stochastic gradient descent learning algorithm~\eqref{eq:Atl} with $\lambda=0$. Suppose $\exists \alpha>0$ such that $\E[X_0X_0^TQX_0X_0^T] \leq \alpha \Sigma_0 Q\Sigma_0$ for all symmetric matrices $Q$, and
$\|A^{(k)}\|_{L^2}\leq M$ for all $k \in \mathbb{N}$. Then there exists a positive constant $\beta>0$ such that $\sfJ$ is a $\beta$-smooth function. And 
for sufficiently small constant stepsize $\eta_k=\eta\leq \frac{1}{\alpha\beta}$, 
\[
\sfJ[A^{(k)}]-\sfJ^* \leq (1-\frac{\eta}{2})^k(\sfJ[A^{(0)}]-\sfJ^*) + \eta\beta e^{2M\sqrt{T}} \E[|X_0|^2|\noise|^2],
\]
for all $k\in \mathbb{N}$ where $\sfJ^*:=\min_A \sfJ[A]=E[|\noise|^2]$.
\label{prop:learning}
\end{proposition}
\begin{proof}
The proof is based on Theorem~4.8 in \cite{bottou2016} where it is shown that SGD converges to a local minimum. To apply the theorem we show 
\begin{align*}
\E\left[Y_t^{(k)}{X_t^{(k)}}^\top\right] &= \nabla \sfJ[A^{(k)}]
\end{align*}
because  $X^{(k)}_0$ is a random sample of $X_0$ and $\nabla \sfJ[A^{(K)}]$ is given by the formula~\eqref{eq:gradJ} for $\lambda=0$. Next 
\begin{align*}
\E\left[\left\|Y_t^{(k)}{X_t^{(k)}}^\top
\right\|^2\right] 
&=\E\left[\|\phi({T,t})^\top(Z-X^i_\tf){X_0^i}^\top\phi({t,0})^\top\|^2\right]\\
&=\E\left[\|\phi({T,t})^\top(R-\phi({T,0}))X_0X_0^\top\phi({t,0})^\top\|^2\right] + \E\left[\|\phi({T,t})^\top WX_0^\top\phi({t,0})^\top\|^2\right] \\
&\leq \alpha\|\phi({T,t})^\top(R-\phi({T,0}))\Sigma\phi({t,0})^\top\|^2 + e^{2 \int_0^\tf \|A_t\|\ud t}\E[|W|^2]\E[|X_0|^2]
\end{align*}
where the assumption $\E[X_0X_0^\top \phi({t,0})^\top\phi({t,0})X_0X_0^\top]\leq \alpha \Sigma \phi({t,0})^\top\phi({t,0}) \Sigma$ and Lemma~\ref{lem:ineq} is used. 

The fact that $\sfJ$ is $\beta$-smooth is true since all the functions involved are smooth and it is assumed  $A$ is bounded. Applying Theorem~4.8 in \cite{bottou2016}, SGD algorithm converges to a local minimum. 
The geometric convergence to the global minimum follows from Theorem~\ref{prop:critical-pt}
where it is shown that local minimum are global minimum for $\lambda=0$ and using the inequality~\eqref{eq:gradJ-J}.  
\end{proof}
\subsection{Proof of Proposition~\ref{prop:opt-vec}}\label{apdx:opt-vec}
Suppose $(X_t,Y_t,A_t)$ is a solution of the Hamilton's equations\eqref{eq:Xt}-\eqref{eq:At}. Then by differentiating $A_t$ with respect to $t$, one obtains  
\begin{equation*}
\frac{\ud A_t}{\ud t} = -A_t^\top A_t + A_tA_t^\top 
\end{equation*}
On expressing $A_t=S_t + \Omega_t$ as the sum of its symmetric component $S_t=\frac{1}{2}(A_t+A_t^\top)$ and the skew-symmetric component $\Omega_t=\frac{1}{2}(A_t-A_t^\top)$, one obtains
\begin{align*}
\frac{\ud S_t}{\ud t} &= 2\Omega_tS_t - 2S_t\Omega_t, \quad
\frac{\ud \Omega_t}{\ud t} = 0 
\end{align*}
whose solution is given by
\begin{align*}
S_t &= e^{2t\Omega}S_0e^{-2t\Omega},\quad
\Omega_t = \Omega_0 
\end{align*}
This gives~\eqref{eq:opt-At-vec}.

Using the formula~\eqref{eq:opt-At-vec} for $A_t$, the Hamilton's equation for $X_t$ is
\begin{equation*}
\frac{\ud X_t}{\ud t} = e^{2t\Omega}Se^{-2t\Omega}X_t + \Omega X_t
\end{equation*} 
whose solution is given by~\eqref{eq:XtC}.

The optimal costate trajectory is obtained similarly. The Hamilton's equation for the costate is: 
\begin{equation*}
\frac{\ud Y_t}{\ud t} = -e^{2t\Omega}Se^{-2t\Omega}Y_t + \Omega Y_t,\quad Y_T=Z-X_T
\end{equation*}
whose solution is given by~\eqref{eq:YtC}. 

The characteristic equation~\eqref{eq:char-vec-A} is obtained by using the formula $A_t = \frac{1}{\lambda}\expect{Y_tX_t'}$:
\begin{equation*}
\lambda e^{2t\Omega}\A0e^{-2t\Omega} = e^{2 t\Omega} e^{-t\A0}e^{\tf\A0}e^{-2\tf \Omega} \expect{(Z-X_\tf)X_0^\top} e^{t\A0}e^{-2t\Omega}
\end{equation*}
upon multiplying both sides from left by $ e^{t\A0}e^{-2t\Omega} $ and from right by $e^{2 t\Omega} e^{-t\A0}$.

\newP{Optimal cost} Optimal cost is obtained by inserting $A_t = e^{2t\Omega} \A0 e^{-2t\Omega}$ into the cost function where the following identities are used:
\begin{align*}
\tr(A_tA_t^\top) &= \tr(\A0\A0^\top)\\
\expect{|X_\tf - Z|^2}&= \expect{|W|^2} + \expect{|FX_0-RX_0|^2}
\end{align*}

\newP{Constant $\Leftrightarrow$ normal} 
Suppose $A_t=C$ a constant. Then $\frac{\ud A_t}{\ud t}=-A_t^\top A_t + A_tA_t^\top =0$, and hence $A_t=C$ is a normal matrix. Conversely, assuming $A_t$ is a normal matrix implies $\frac{\ud A_t}{\ud t} = 0$ and hence $A_t=C$ a constant.  

\newP{Normal solution}
If $\A0$ is normal, then $\A0$ and $\Omega$ commute, therefore $F =e^{\tf \A0}$. Hence the characteristic simplifies to
\begin{equation*}
\lambda \A0 = e^{\A0^\top}(R-e^{\A0})\Sigma
\end{equation*}
and equivalently
\begin{equation*}
\lambda \A0e^{-\A0^\top}\Sigma^{-1} +e^{\A0} = R
\end{equation*}
Therefore, if $\A0$ and $\Sigma$ commute (always true when $\Sigma=I$), $R$ is a normal matrix. We have proved
\begin{equation*}
A_t \equiv \A0 ~~\Longleftrightarrow~~ \A0 ~~\text{is normal} ~~\overset{(\Sigma=I)}{\Longrightarrow}~~R~~\text{is normal}
\end{equation*}
Therefore a non-normal $R$ implies the minimizer $A_t$ is not constant for $\Sigma=I$. 
\subsection{Proof of Theorem~\ref{cor:opt-vec}}\label{apdx:cor-opt-vec}
\begin{enumerate}
\item If $\A0$ is normal, then $e^{\tf(\A0-\A0^\top)}e^{\tf \A0}=e^{\tf \A0}$. Hence for $\lambda=0$ problem the characteristic equation becomes $e^{\tf \A0}=R$ whose solution is $\A0=\frac{1}{\tf}\log(R)$, interpreted as multi-valued matrix logarithm function (see~\cite{higham2014functions}).

\item For $\lambda>0$ and $\Sigma=I$ the characteristic equation is:
\begin{equation*}
\lambda \A0e^{-\A0^\top} +e^{\A0} = R
\end{equation*}
Since $\A0$ is normal, $R$ must be normal and moreover there exists a unitary (complex) matrix $U$ such that $U^*RU=D$ where $D = \text{diag}(r_1,\ldots,r_d)$ with $r_n \in \mathbb{C}$. Let $\mu_n \in \mathbb{C}$ be solution to the equation  
\begin{equation}
\lambda \mu_ne^{-\mu_n^*}+e^{\mu_n} = r_n
\label{eq:mu}
\end{equation}  
for $n=1,\ldots,d$. Then $\A0=UGU^*$ where $G=\text{diag}(\mu_1,\ldots,\mu_d)$ is the normal solution to the characteristic equation since 
\begin{equation*}
\lambda Ge^{-G^*} +e^{G}=D \quad \Rightarrow \quad \lambda UGe^{-G^*}U^* +Ue^{G}U^*=UDU^* \quad \Rightarrow \quad \lambda \A0 e^{\A0^\top} + e^\A0 = R
\end{equation*}
%
It thus suffices to analyze solutions to the complex equation~\eqref{eq:mu}. 
Denoting $\mu_n=x + iy$ and $r_n=e^{a+i\theta}$ the complex equation~\eqref{eq:mu} is written as two real equations:
\begin{align*}
f_1(x,y;\lambda):=\lambda xe^{-x} \cos(y) -\lambda ye^{-x}\sin(y) +e^x \cos(y)&=e^a\cos(\theta)\\
f_2(x,y;\lambda):=\lambda xe^{-x}\sin(y) + \lambda ye^{-x}\cos(y)+e^x\sin(y) &= e^a\sin(\theta)
\end{align*}
At $\lambda=0$, there are countability many solutions given by $x_0 = a$ and $y_0=\theta + m2\pi$ for $m \in \mathbb{Z}$. 
The Jacobian 
\begin{equation*}
\Df(x_0,y_0;0)=
\begin{bmatrix}
\frac{\partial f_1}{\partial x}(x_0,y_0,0) & \frac{\partial f_1}{\partial y}(x_0,y_0,0) \\
\frac{\partial f_2}{\partial x}(x_0,y_0,0) &\frac{\partial f_2}{\partial y}(x_0,y_0,0) 
\end{bmatrix}=
\begin{bmatrix}
e^{x_0} \cos(y_0) & -e^{x_0}\sin(y_0) \\
e^{x_0} \sin(y_0) & e^{x_0} \cos(y_0)
\end{bmatrix}
\end{equation*}
is nonsingular since $\det(\Df) = e^{2x_0} = e^{2a} >0$. 
Therefore, using the implicit function theorem, there exists a neighborhood $\cal{N}$ of $\lambda=0$ and a function $\lambda\in \mathcal{N} \to (x(\lambda),y(\lambda))\in \Re^2$ such that $f(x(\lambda),y(\lambda);\lambda)=0$. The asymptotic formula for $x(\lambda)$ and $y(\lambda)$ are obtained upon using a regular perturbation expansion   
$x=x_0+\lambda x_1 + O(\lambda^2)$ and $y=y_0+\lambda y_1 + O(\lambda^2)$. Then 
\begin{align*}
\begin{bmatrix}
x_1 \\
y_1
\end{bmatrix}
&=-[\Df(x_0,y_0;0)]^{-1} 
\frac{\partial f}{\partial \lambda}(x_0,y_0;0)
\\
&=-e^{-x_0}\begin{bmatrix}
 \cos(y) & \sin(y_0) \\
-\sin(y) &  \cos(y_0)
\end{bmatrix}\begin{bmatrix}
x_0e^{-x_0} \cos(y_0) -y_0e^{-x_0}\sin(y_0) \\
x_0e^{-x_0}\sin(y_0) + y_0e^{-x_0}\cos(y_0)
\end{bmatrix}
\\ &=-e^{-2 x_0}\begin{bmatrix}
x_0 \\
y_0
\end{bmatrix}
\end{align*}
Therefore
\begin{equation*}
\mu = \log(r) - \lambda \frac{\log(r)}{|r|^2} + O(\lambda^2)
\end{equation*}
which gives the asymptotic formula
\begin{equation*}
\A0 = \log(R) -\lambda (RR^*)^{-1}\log(R) + O(\lambda^2)
\end{equation*}
\end{enumerate}
\end{document}